\documentclass[12pt]{amsart}
\usepackage{amsmath,amssymb,amsfonts,amsthm,amsopn,dsfont}
\usepackage[dvips]{graphicx}
\usepackage{srcltx}
\usepackage{setspace}
\newtheorem{theorem}{Theorem}[section]

\numberwithin{equation}{section}

\newtheorem{proposition}[theorem]{Proposition}

\renewcommand{\ell}{l}
\renewcommand{\epsilon}{\varepsilon}

\def\ga{\gamma}

 \def\cS{\mathcal{S}}

\def\px{\langle x \rangle}

\def\rd{\bR^d}

\newcommand{\Fur}{\mathcal{F}}

\def\R{\right)}

\def\<{\left<}
\def\>{\right>}

\def\mv1{M_v^1}

\def\mn{(m,n)}
\def\mn'{(m',n')}

\newcommand{\bC}{\mathbb{C}}
\def\N{\mathbb{N}}
\def\R{\mathbb{R}}
\def\C{\mathbb{C}}
\def\rd{\mathbb{R}^d}

\begin{document}

\title[]{Sharp decay estimates and smoothness for solutions to nonlocal semilinear equations}
\begin{abstract}
We consider semilinear equations of the form $p(D)u=F(u)$, with a locally bounded nonlinearity $F(u)$, and a linear part $p(D)$ given by a Fourier multiplier. The multiplier $p(\xi)$ is the sum of positively homogeneous terms, with at least one of them non smooth.  This general class of equations includes most physical models for traveling waves in hydrodynamics, the Benjamin-Ono equation being a basic example.\par
We prove sharp pointwise decay estimates for the solutions to such equations, depending on the degree of the non smooth terms in $p(\xi)$. When  the nonlinearity is smooth we prove similar estimates for the derivatives of the solution, as well holomorphic extension to a strip, for analytic nonlinearity.
\end{abstract}
\author{Marco Cappiello \and Fabio Nicola}
\address{Dipartimento di Matematica,  Universit\`{a} degli Studi di Torino,
Via Carlo Alberto 10, 10123
Torino, Italy
} \email{marco.cappiello@unito.it}
\address{Dipartimento di Scienze Matematiche, Politecnico di
Torino, Corso Duca degli
Abruzzi 24, 10129 Torino,
Italy}
\email{fabio.nicola@polito.it}
\subjclass[2000]{}
\date{}
\maketitle

\section{Introduction}
The main goal of this paper is to derive sharp pointwise decay estimates and holomorphic extensions for the solutions of a class of nonlocal semilinear equations in $\R^d$. Namely we consider an equation of the form
\begin{equation} \label{equation}
p(D)u=F(u),
\end{equation}
where $F:\bC\to\bC $ is any measurable function satisfying 
\begin{equation}\label{nonlin}
|F(u)|\leq C_K |u|^{p}
\end{equation}
for some $p>1$, uniformly for $u$ in compact subsets $K\subset\bC$. For example, $F$ could be any smooth function vanishing to second order at the origin. 
Concerning the linear part, we assume that $p(D)$ is a Fourier multiplier with symbol of the form
\begin{equation}\label{simbolo}
p(\xi)=p_0+\sum_{j=1}^h p_{m_j}(\xi)
\end{equation}
where $p_0 \in\C$ and $p_{m_j}(\xi)$ are smooth in $\rd\setminus\{0\}$ and positively homogeneous of degree $m_j$, with $0<m_1<m_2<\ldots<m_h=M$.\par
We suppose that $p(\xi)$ satisfies the condition
\begin{equation}\label{ellitticita}
|p(\xi)|\geq C\langle \xi\rangle^M,\quad \xi\in\rd,
\end{equation}
for some $C>0$ (as usual $\langle \xi\rangle=(1+|\xi|^2)^{1/2}$).
This implies in particular that $p_0\not=0$.\par
 Moreover we suppose that there is at least a symbol $p_{m_j}$ which is non smooth, and we define the singularity index
\begin{equation}
m=\min\{m_j:\ p_{m_j}\text{ is not smooth}\},
\end{equation}
so that $m >0$.
\par
Equations of the form above appear in a big number of physical models, especially in hydrodynamics, where they are related to the existence of travelling wave type solutions $v(t,x)=u(x-ct)$, $c>0$ for an evolution equation of the form 
\[
v_t + (p(D)v)_x =F(v)_x,\quad t \geq 0,\ x \in \R.
\]
 In the case when $p(\xi)$ is a uniformly analytic symbol in $\R$, Bona and Li \cite{BL1, BL2} proved that every solution of \eqref{equation} exhibits an exponential decay for $|x| \to \infty$, therefore of the form $e^{-c|x|}$ for some $c>0$, and admits an extension to a holomorphic function in a strip of the form $\{z \in \mathbb{C}: |{\rm Im}\, z| <T\}$ for some $T>0$. The results in \cite{BL1, BL2} apply in particular to KdV-type equations, long-wave-type equations and certain Schr\"odinger-type equations. Later on the results above have been extended in arbitrary dimension by the first author et al. to more general equations with linear part given by an analytic pseudodifferential operator, see \cite{A18, A39, CGR, CGR2}. More recently, the results on the holomorphic extensions have been refined in \cite{CN1, CN2}. \par
Concerning the case when the symbol $p(\xi)$ is only finitely smooth at $\xi =0$, many models suggest that the exponential decay observed in the previous papers is lost in general and replaced by an algebraic decay at infinity depending on the singularity index $m$ and on the dimension $d$.
The most celebrated example in this connection is the Benjamin-Ono equation studied by Benjamin \cite{Benjamin} and Ono \cite{Ono} (see also \cite{BP, LPP, MST, T}) and describing the propagation of one dimensional internal waves in stratified fluids of great depth:
\begin{equation} \label{BO}
\partial_t v +H(\partial_x^2 v)+2v \partial_x v =0, \quad t \in \R,\ x \in \R,
\end{equation}
where $H(D)$ denotes the Hilbert transform. When looking at travelling solutions of \eqref{BO} one is reduced to study the equation
\begin{equation} \label{BO2}
|D|u+u =u^2,
\end{equation} which is of the form \eqref{equation} with $p(\xi)=1+|\xi|$ and $F(u)=u^2,$ with singularity index $m=1$. It is well known, see e.g. \cite{AT}, that the equation \eqref{BO2} admits the solution $$u(x)= \frac1{1+x^2}.$$
Maris \cite{M} considered the following generalization of \eqref{BO2} in higher dimension:
\begin{equation} \label{genBO2}
(1+(-\Delta)^{1/2})u=F(u), \end{equation}
where $F$ satisfies the same assumption as in the present paper. He proved that every solution $u$ of \eqref{genBO2} which is in $L^\infty(\R^d)$ and tends to $0$ for $|x| \to \infty$, 
 actually decays like $|x|^{-d-1}$.\\
In a very recent paper \cite{CGRN}, the first author et al. considered the general equation \eqref{equation} with $p(\xi)$ as in \eqref{simbolo} and $F(u)$ given by a polynomial of degree at least $2$ in $u$ and they derived $L^2$ decay estimates for weak Sobolev type solutions of \eqref{equation}. Namely,  they proved that every solution $u$ of \eqref{equation} such that $\px^{\varepsilon_o}u \in H^s (\R^d)$ for some $s >d/2$ and $\varepsilon_o >0$ satisfies indeed the following weighted estimates:
\begin{equation}
\label{weighted}
\| \px^{m+d/2+|\alpha|-\varepsilon} \partial^\alpha u \|_s <\infty.
\end{equation}
Unfortunately, the latter result has been proved under the assumption $[m] >d/2$, which seems indeed to be only technical but very restrictive, especially in high dimension. Moreover, pointwise decay estimates remained out of reach in \cite{CGRN}. This was due to the fact that the argument of the proof relies in particular on the H\"ormander-Mihlin theorem for Fourier multipliers which does not apply on $L^{\infty}(\R^d)$.
Nevertheless the estimate \eqref{weighted} together with the results in \cite{M} and some examples reported in \cite{CGRN} lead to conjecture that the solution of \eqref{equation} may satisfy a pointwise estimate of the form 
\begin{equation}\label{pointwise}
\| \langle x\rangle ^{m+d} u\|_{L^\infty} < \infty. 
\end{equation}
\par
In the present paper we prove this, together with higher regularity estimates, when the nonlinearity is smooth or analytic.\par
Namely, we have the following results.
\begin{theorem}\label{mainteo}
With the above notation, assume \eqref{nonlin}, \eqref{simbolo}, \eqref{ellitticita}. Let $u$ be a distribution solution of $p(D)u=F(u)$, satisfying $\langle x\rangle ^{\epsilon_0} u\in L^\infty(\R^d)$ for some $\epsilon_0>0$. Then $\langle x\rangle ^{m+d} u\in L^\infty (\R^d)$. 
\end{theorem}
When the nonlinearity $F$ is smooth, similar estimates hold for the derivatives.
\begin{theorem}\label{mainteo2}
Assume \eqref{simbolo}, \eqref{ellitticita} and let $F\in C^\infty(\mathbb{C})$, with $F(0)=0$, $F'(0)=0$.
Let $u$ be a distribution solution of $p(D)u=F(u)$, satisfying $\langle x\rangle ^{\epsilon_0} u\in L^\infty(\rd)$ for some $\epsilon_0>0$.\par
 Then $u$ is smooth and satisfies the estimates
\begin{equation}\label{stime2}
\|\langle x \rangle^{m+d+|\alpha|}\partial^\alpha u\|_{L^\infty}<\infty,\quad \alpha\in\mathbb{N}^d.
\end{equation}
\end{theorem}
Here we mean that $F$ is smooth with respect to the structure of $\C$ as a real vector space, and $F'$ denotes its differential. Observe, in particular, that the above assumptions on $F$ imply \eqref{nonlin}.\par
As an immediate consequence of Theorem \ref{mainteo2}, we derive $L^p$ weighted estimates for the solution of \eqref{equation} also for $p \in [1,+\infty).$ In particular, for $p=2$ we recapture the results proved in \cite{CGRN} under the stronger assumption $[m]>d/2$ (whereas here we just require $m>0$).
\begin{theorem} \label{mainteo2bis}
Let $ p \in [1, +\infty)$. Under the assumption of Theorem \ref{mainteo2}, the solution $u$ of \eqref{equation} satisfies the estimates
\begin{equation}\label{stime3}
\|\langle x\rangle^{m+d(1-1/p)+|\alpha|-\varepsilon}\partial^\alpha u\|_{L^p}<\infty,\quad \alpha\in\mathbb{N}^d,
\end{equation}
for every $\varepsilon >0.$
\end{theorem}

The examples in Section 2 below show that the results above are completely sharp, both for the decay of $u$ and for that of its derivatives.\par
Finally, in the case of polynomial non-linearity one can also prove analytic regularity for the solution, in analogy with \cite{A3, BL1, BL2, A18, A39, CGR, CGR2}.

\begin{theorem}\label{mainteo3}
Assume \eqref{simbolo}, \eqref{ellitticita} and let $F$ be a polynomial in $u,\overline{u}$, with $F(0)=0$, $F'(0)=0$.
Let $u$ be a distribution solution of $p(D)u=F(u)$, satisfying $\langle x\rangle ^{\epsilon_0} u\in L^\infty(\R^d)$ for some $\epsilon_0>0$.\par
 Then there exists $\epsilon>0$ such that $u$ extends to a bounded holomorphic function $u(x+iy)$ in the strip $\{z=x+iy\in\mathbb{C}^d:\,|y|<\epsilon\}$.  
\end{theorem}
Actually the estimates \eqref{stime2} suggest a stronger result, namely that $u$ extends to a holomorphic function to a sector. This is also confirmed by the examples in Section 2, but the proof seems to be extremely technical, cf. \cite{CN1, CN2}, and therefore we prefer to postpone this issue to a future paper.

\section{Examples}
In this section we provide some examples showing the sharpness of our decay estimates. 
First of all we recall the following formula for the Fourier transform of functions of the form $(1+|x|^2)^{-\lambda}, $ $ \lambda >0$ (see for example \cite{Schwartz}, formula (VII, 7;23), page 260, or \cite{GS1}, formula (9), page 187).
\begin{equation}\label{F}
\mathcal{F}((1+|x|^2)^{-\lambda})(\xi)= \frac{2\pi^{d/2}}{\Gamma(\lambda)} \left( \frac{|\xi|}{2}\right)^{\lambda-\frac{d}{2}} K_{\lambda-\frac{d}{2}}(|\xi|),
\end{equation}
where $x,\xi \in \R^d$, $\Gamma$ denotes the standard Euler Gamma function and $K_\nu(x), \nu \in \R, x \in \R \setminus 0$, are the modified Bessel functions of second type, see \cite{Erd, Watson} for definitions and properties. Here we just recall that
\begin{equation} \label{K1}
K_\nu(x)= K_{-\nu}(x), \qquad \nu \in \R, x \neq 0,
\end{equation}
\begin{equation} \label{K2}
K_{\nu+1}(x)=\frac{2\nu}{x}K_\nu(x) + K_{\nu-1}(x), \qquad \nu \in \R, x \neq 0.
\end{equation}

Using the properties \eqref{F}, \eqref{K1}, \eqref{K2}, an alternative proof of the fact that the function $$u(x)= \frac1{1+x^2}$$ is a solution of the Benjamin-Ono equation \eqref{BO2} has been given in \cite{CGRN}. We repeat here the proof for the benefit of the reader since it is very short. In fact we have 

\begin{align} \label{2.15}
\mathcal{F}(u^2)&=4\mathcal{F}((1+x^2)^{-2})= 8\sqrt{\pi}\left(\frac{|\xi|}{2}\right)^{3/2}K_{\frac{3}{2}}(|\xi|) \\
&=4\sqrt{\pi}(|\xi|+1)\left(\frac{|\xi|}{2}\right)^{1/2}K_{\frac{1}{2}}(x)\nonumber \\ &= 2(|\xi|+1) \mathcal{F}((1+x^2)^{-1}) \nonumber \\ &= \mathcal{F}(|D|u+u).\nonumber
\end{align}

Using \eqref{F} many other examples of equations of the form \eqref{equation}, \eqref{simbolo} admitting solutions of the form $(1+|x|^2)^{-\lambda}$ for some $\lambda >0$, can be constructed. We provide here a couple of such examples. The first one is taken from the unpublished manuscript \cite{GVS}. We describe this model in detail for the sake of completeness.
\\

\noindent \textbf{Example 1.} Consider  in dimension $d=1$ the equation
\begin{equation}
|D|^3+6D^2 u + 15|D|u+15u =48u^4
\end{equation}
which is of the form \eqref{equation}, \eqref{simbolo} with $m=1, d=1$. Moreover the condition \eqref{ellitticita} is satisfied. Using \eqref{F}, \eqref{K1}, \eqref{K2} it is easy to prove that it admits the solution
$$ u(x)=\frac{1}{1+x^2}, \qquad  x\in \R.$$
As a matter of fact we have
\begin{align}\label{Fex} K_{\frac72}(x)&= \frac{5}{x}K_{\frac52}(x)+K_{\frac32} (x)\\
&=  \frac{5}{x} \left( \frac{3}{x^2}+\frac{3}{x}+1 \right)K_{\frac{1}{2}}(x) +\left(\frac{1}{x}+1 \right)K_{\frac12}(x) 
 \nonumber \\ &= \left( \frac{15}{x^3}+\frac{15}{x^2}+\frac{6}{x}+1 \right) K_{\frac12}(x), \qquad x \neq 0. \nonumber
\end{align}
Then from \eqref{F} and \eqref{Fex}
\begin{align*}
48\mathcal{F} (u^4) &= 48 \mathcal{F}((1+x^2)^{-4}) = \frac{96\sqrt{\pi}}{\Gamma(4)} \left( \frac{|\xi|}{2}\right)^{\frac72} K_{\frac72}(|\xi|) \\ 
&= 2\sqrt{\pi} (|\xi|^3+6|\xi|^2+15|\xi|+15) \left( \frac{|\xi|}{2}\right)^{\frac12}K_{\frac12}(|\xi|) \\
&= (|\xi|^3+6|\xi|^2+15|\xi|+15) \mathcal{F}((1+x^2)^{-1})\\
&= \mathcal{F}(|D|^3 u +6D^2 u+15|D|u+15u).
\end{align*}
The example above, compared with the Benjamin-Ono equation, confirms that the presence of terms with finite but stronger regularity than $|\xi|^m$ in the symbol of the linear part do not have any influence on the decay of the solution, which is determined only by the singularity index.

The next example in dimension $d=3$ has been proposed in \cite{CGRN} as a model which did not satisfy the condition $[m]>d/2$ required in \cite{CGRN} but exhibiting the right decay properties as well. We recall it here since our weaker condition $m>0$ allows to include it completely in our results.
\\

\noindent \textbf{Example 2.} Consider the equation
\begin{equation} \label{tridim}
-\Delta u +3 (-\Delta)^{1/2}u +3u =24 u^2,
\end{equation}
where we denote by $(-\Delta)^{1/2}=|D|$ the Fourier multiplier with symbol $|\xi|, \xi \in \R^3$. The linear part of \eqref{tridim} satisfies the condition \eqref{ellitticita} and the singularity index is $m=1$. Moreover the function
$$u(x)= \frac{1}{(1+|x|^2)^2}, \qquad x \in \R^3$$
is a solution of \eqref{tridim}. In fact we have
\begin{align*}
24 \mathcal{F}(u^2) &= 24 \mathcal{F}((1+|x|^2)^{-4})\\
&= 8\pi^{\frac{3}{2}} \left( \frac{|\xi|}{2}\right)^{\frac{5}{2}}K_{\frac{5}{2}}(|\xi|) \\
&= (|\xi|^2+3|\xi|+3) \mathcal{F}((1+|x|^2)^{-2})\\
&= \mathcal{F}(-\Delta u +3 (-\Delta)^{1/2}u+3u). 
\end{align*}

\section{Notation and preliminary results}
For $s\geq0$, define the weight functions
\[
v_s(x)=\langle x\rangle^s,\ x\in\rd,
\]
and the weighted Lebesgue spaces
\[
L^\infty_{v_s}=\{u\in L^\infty(\rd):\ \|u\|_{L^\infty_{v_s}}:=\|v_s u\|_{L^\infty}<\infty\}.
\]
We will use often the following well-known interpolation inequality.
\begin{proposition}\label{interp}
Given $0\leq \ell\leq n$, there exists a constant $C>0$ such that the following inequality holds.
Let $I=I_1\times\ldots\times I_d\subset\rd$, where each $I_j$, $j=1,\ldots,d$, is an interval of the form $[a,+\infty)$ or $(-\infty,a]$. Then
\begin{equation}\label{inter}
\|D^\ell u\|_{L^\infty(I)}\leq C\|u\|_{L^\infty(I)}^{1-\ell/n}\|D^n u\|_{L^\infty(I)}^{\ell/n},
\end{equation}
where we set $\|D^k u\|_{L^\infty(I)}:=\sup_{|\alpha|=k}\|\partial^\alpha u\|_{L^\infty(I)}$, $k\in\mathbb{N}$. \par
In particular, if $u\in L^\infty_{v_s}$ and $D^n u\in L^\infty_{v_r}$ then $D^l u\in L^\infty_{v_\nu}$ with $\nu=(1-\ell/n)s+(\ell/n)r$.

\end{proposition}
\begin{proof}
In the case $d=1$, $I=[0,+\infty)$ or $I=(-\infty,0]$, formula \eqref{inter} is known as Landau-Kolmogorov inequality. By Young inequality we also have 
\[
\|D^\ell u\|_{L^\infty(I)}\leq C\|u\|_{L^\infty(I)}+C\|D^n u\|_{L^\infty(I)},
\]
in that case. A repeated application of this inequality shows that it also holds in the multi-dimensional case, with $I$ as in the statement and each $I_j$ of the form $[0,+\infty)$ or $(-\infty,0]$. Now, we can apply this estimate to $u(\epsilon x)$, $\epsilon>0$ and optimize in $\epsilon$ (cf.\ \cite[Proposition 3.4]{taylor2}), obtaining the desired conclusion, namely \eqref{inter}, for sets $I$ of this special form. The general case in the statement can be obtained simply by a translation. \par
For the last statement, one can estimate, for $|\alpha|=\ell$,
\[
\langle x\rangle^{\nu}|\partial^\alpha u(x)|\leq \langle x\rangle^{\nu} \|D^\ell u\|_{L^\infty(I)}
\]
when $x=(x_1,\ldots,x_d)$, $x_1\geq 0,\ldots x_d\geq 0$, and $I=[x_1,+\infty)\times\ldots\times[x_d,+\infty)$ and then apply \eqref{inter}. If e.g. $x_1<0$ one considers instead the interval $(-\infty,x_1]$, and similarly for the other components. 
\end{proof}

We also define the space 
\[
\Fur L^\infty_{v_s}=\{u\in S'(\rd):\ \|u\|_{\Fur L^\infty_{v_s}}:=\|\widehat{u}\|_{L^\infty_{v_s}}<\infty\},
\]
where
\[
\Fur u(\xi)=\widehat{u}(\xi)=\int_{\rd} e^{-ix\xi}u(x)\, dx
\]
is the Fourier transform of $u$.\par
We recall the following result (see e.g. \cite[Proposition 11.1.3 (b)]{book}).
\begin{proposition}\label{pro0}
If $s>d$, the space $L^\infty_{v_s}$ is a convolution algebra and 
\[
\|u\ast v\|_{L^\infty_{v_s}}\leq C_s \|u\|_{L^\infty_{v_s}}\|v\|_{L^\infty_{v_s}}
\] 
for some constant $C_s>0$.\par
As a consequence, $\Fur L^\infty_{v_s}$ is an algebra with respect to pointwise multiplication, and 
\[
\|u v\|_{\Fur L^\infty_{v_s}}\leq C_s \|u\|_{\Fur L^\infty_{v_s}}\|v\|_{\Fur L^\infty_{v_s}}.
\] 
\end{proposition}
Possibly after passing to equivalent norms we can suppose that the above inequalities hold with $C_s=1$, so that $\Fur L^\infty_{v_s}$ becomes a Banach algebra.\par
We will use the following easy criterion of continuity on $L^\infty_{v_s}$ for a convolution operator. We report on the proof for the benefit of the reader.
\begin{proposition}\label{pro1}
Any convolution operator with integral kernel $K\in L^\infty_{v_s}$, with $s>d$, is bounded on $L^\infty_{v_r}$ for every $0\leq r\leq s$.
\end{proposition}
\begin{proof} It suffices to prove the following estimate:
\[
\sup_{x\in\rd}\int_{\R^{d}}\frac{1}{\langle x-y\rangle^s}\frac{\langle x\rangle^r}{\langle y\rangle^r}\, dy<\infty.
\]
This is true if $r=0$, because $s>d$. When $r=s(>d)$ the inequality is proved e.g. in \cite[Formula (11.5)]{book}. For the intermediate values of $r$ we split the above integral as follows
\begin{align*}
&\int_{\langle y\rangle\geq\langle x\rangle}\frac{1}{\langle x-y\rangle^s}\underbrace{\frac{\langle x\rangle^r}{\langle y\rangle^r}}_{\leq 1}\, dy+\int_{\langle y\rangle<\langle x\rangle}\frac{1}{\langle x-y\rangle^s}\underbrace{\frac{\langle x\rangle^r}{\langle y\rangle^r}}_{\leq \frac{\langle x\rangle^s}{\langle y\rangle^s}}\, dy\\
&\leq \int_{\mathbf{R}^{d}}\frac{1}{\langle x-y\rangle^s} dy+ \int_{\mathbf{R}^{d}}\frac{1}{\langle x-y\rangle^s}\frac{\langle x\rangle^s}{\langle y\rangle^s}\, dy\leq C.
\end{align*}
\end{proof}

Finally we recall the following result, which follows easily from the theory of homogeneous distributions.
\begin{proposition}\label{pro4}
Let $f\in C^\infty(\rd\setminus\{0\})$ be positively homogeneous of degree $r>0$ and $\chi\in C^\infty_0(\rd)$. There exists a constant $C>0$ such that
\[
|\widehat{\chi f}(\xi)|\leq C(1+|\xi|)^{-r-d},\quad \xi\in\rd,
\]
i.e. $\chi f\in \Fur L^\infty_{v_{r+d}}$.
 \end{proposition}
 \begin{proof}
 We know that the Fourier transform of $f$ is a homogeneous distribution of degree $-r-d$, smooth in $\rd\setminus\{0\}$ \cite[Vol.1, Theorems 7.1.16, 7.1.18]{hormander3}. Hence, if $\chi'\in C^\infty_0(\rd)$, $\chi=1$ in a neighborhood of the origin we have 
 \begin{equation}\label{equl}
 |(1-\chi'(\xi)) \widehat{f}(\xi)|\leq C(1+|\xi|)^{-r-d},\quad \xi\in\rd.
 \end{equation}
 On the other hand, we have 
 \[
 |\widehat{\chi f}(\xi)|\leq(2\pi)^d \Big(|\big((\chi'\widehat{f})\ast\widehat{\chi}\big)(\xi)|+ |\big( \big((1-\chi')\widehat{f}\big)\ast\widehat{\chi}\big)(\xi)|\Big).
  \]
Since $\widehat{\chi}\in\cS(\rd)$, the first term in the right-hand side has a rapid decay, because $\mathcal{E}'\ast\cS\subset\cS$,  whereas the second term can be easily estimated using \eqref{equl}. 
 \end{proof}

\section{Decay of the solution: proof of Theorem \ref{mainteo}}
To prove Theorem \ref{mainteo} it is sufficient to prove the following boundedness result.
\begin{proposition}\label{pro5}
For every $0\leq s\leq m+d$ we have
\[
p(D)^{-1}=p^{-1}(D): L^\infty_{v_s}\to L^\infty_{v_s}
\]
continuously.
\end{proposition}
Indeed, since $u\in L^\infty_{v_{\varepsilon_o}}$  we have $F(u)\in L^\infty_{v_{p\varepsilon_o}}$. Moreover the equation can be written in integral form as
\[
u=p(D)^{-1}(F(u)).
\]
Hence, Proposition \ref{pro5} implies that if $p\varepsilon_o \leq d+m$ then 
$$u=p(D)^{-1}(F(u)) \in L^\infty_{v_{p\varepsilon_o}} \subset L^\infty_{v_{d+m}}.$$

\begin{proof}[Proof of Proposition \ref{pro5}]
Let $\chi\in C^\infty_0(\rd)$, $\chi=1$ in a neighborhood of the origin. We write
\[
\frac{1}{p(\xi)}=\underbrace{\frac{\chi(\xi)}{p(\xi)}}_{q_1(\xi)}+\underbrace{\frac{1-\chi(\xi)}{p(\xi)}}_{q_2(\xi)}.
\]
Let us prove that
\begin{equation}\label{eq1}
q_j(D):L^\infty_{v_s}\to L^\infty_{v_s}, \qquad j=1,2,
\end{equation}
continuously.\par
As for $q_1(D)$ is concerned, since $m>0$, by Proposition \ref{pro1} it suffices to prove that the convolution kernel of $q_1(D)$, which is $\Fur^{-1}(q_1)(x)$, belongs to $L^\infty_{v_{d+m}}$, i.e. $q_1\in \Fur L^\infty_{v_{d+m}}$. Now, let $\tilde{\chi}\in C^\infty_0(\rd)$. It follows from Proposition \ref{pro4} that
\[
\tilde{p}(\xi):= \tilde{\chi}(\xi)\sum_{j=1}^h p_{m_j}(\xi)\in\Fur L^\infty_{v_{d+m}}.
\]
Hence, if $\tilde{\chi}=1$ on the support of $\chi$ we have 
\[
q_1(\xi)=\frac{\chi(\xi)}{p(\xi)}=\frac{\chi(\xi)}{p_0+\tilde{p}(\xi)}=\frac{\chi(\xi)}{p_0}\sum_{\ell=0}^\infty\Big(\frac{-\tilde{p}(\xi)}{p_0}\Big)^\ell\in \Fur L^\infty_{v_{d+m}}
\]
provided this series converges in $\Fur L^\infty_{v_{d+m}}$ (because $\Fur L^\infty_{v_{d+m}}$ is a Banach algebra). Now, this is certainly true if $\|\tilde{p}\|_{\Fur L^\infty_{v_{d+m}}}<|p_0|$, and after a rescaling (which preserves the space $\Fur L^\infty_{v_{d+m}})$ we can suppose that this is the case. Indeed, 
\[
p(\epsilon \xi)=p_0+\sum_{j=1}^h \epsilon^{m_j}p_{m_j}(\xi),
\]
but now the terms of order $>0$ are all multiplied by positive powers of $\epsilon$ ($m_j>0$), whereas the constant term $p_0$ remains the same. Hence, if $\epsilon$ is small enough, so that $\chi(\xi)\chi(\epsilon \xi)=\chi(\xi)$, we have 
\[
q_1(\epsilon \xi)=\frac{\chi(\epsilon \xi)}{p(\epsilon \xi)}=\frac{\chi(\xi)}{p_0+\sum_{j=1}^h \epsilon^{m_j}p_{m_j}(\xi)}+\frac{(1-\chi(\xi))\chi(\epsilon \xi)}{p(\epsilon \xi)}.
\]
The first term can now be treated by a power series expansion, as above, whereas the second one is in $C^\infty_0(\rd)\subset\Fur L^\infty_{v_{m+d}}$.
This concludes the proof of \eqref{eq1} for $q_1(D)$.\par
Let us now prove it for $q_2(D)$. Because of the condition \eqref{ellitticita} and the presence of the cut-off we see that $q_2$ belongs to the standard symbol class $S^{-M}$. Since $M>0$, by classical results we have that the convolution kernel $K=\Fur^{-1}(q_2)$ of $q_2(D)$ is a locally integrable function (\cite[Remark (i), pag. 245]{stein}) and that it satisfies the estimate
\begin{equation}\label{stime}
|K(x)|\leq\begin{cases}
C|x|^{-d+M}& |x|\leq 1\\
C_N |x|^{-N}& |x|\geq 1
\end{cases}
\end{equation}
for every $N$; see e.g. \cite[Formulas (0.2.5), (0.5.5)]{taylor}.\par Now, \eqref{eq1} for $q_2$ follows from the estimate
\[
\sup_{x\in \rd}\int_{\rd} |K(x-y)|\frac{\langle x\rangle^s}{\langle y\rangle^s}\, dy<\infty.
\]
This holds true because, using \eqref{stime}, we have
\begin{align*}
&\int_{|x-y|<1}|x-y|^{-d+M}\frac{\langle x\rangle^s}{\langle y\rangle^s}\,dy+\int_{|x-y|\geq 1}\frac{1}{|x-y|^N}\frac{\langle x\rangle^s}{\langle y\rangle^s}\, dy\\
\lesssim&\int_{|x-y|<1}|x-y|^{-d+M}\,dy+\int_{|x-y|\geq 1}\frac{1}{|x-y|^{N-s}}\, dy<C.
\end{align*}
\end{proof}
\section{Smoothness and decay of the derivatives: proof of Theorem \ref{mainteo2}}
First we show that
$u\in H^s $ for all $s\in\R.$
This is obtained by an easy bootstrap argument. Namely, we write the equation as 
\[
u=p(D)^{-1}(F(u)).
\]
Since we know from Theorem \ref{mainteo} that $u\in L^2\cap L^\infty$, we have $F(u)\in L^2$. Now, the multiplier $p(D)^{-1}$ has a symbol $p(\xi)^{-1}$, which is equivalent to $\langle \xi\rangle^{-M}$, by the assumption \eqref{ellitticita}. Hence, by Parseval equality we obtain that $p(D)^{-1}:H^s\to H^{s+M}$, $s\in\R$, continuously. This gives $u\in H^M\cap L^\infty$. One can apply this argument repeatedly, because $u\in H^s\cap L^\infty$ implies that $F(u)\in H^s\cap L^\infty$ by Schauder's estimates. Hence $u \in H^{s+M}$. Iterating this argument we obtain that $u \in H^s$ for all $s \in \R$.\par
Let us now prove \eqref{stime2} by induction on $|\alpha|$. By Theorem \ref{mainteo} it holds for $\alpha=0$. Suppose that it holds for $|\alpha|=N-1$. It will then suffice to prove that
\begin{equation}\label{stima3}
\|\langle x\rangle^{m+d}x^\beta\partial^\alpha u\|_{L^\infty}<\infty,\quad|\beta|\leq |\alpha|=N\geq 1.
\end{equation}
First of all we observe a preliminary (non-optimal) decay for $\partial^\alpha u$. To simplify the notation, set $X^s=L^\infty_{v_s}$, $s\geq0$. Since $u \in H^s$ for any $s \in \R$, from the inductive hypothesis and the interpolation inequalities (Proposition \ref{interp} with $\ell=1$, $n$ large enough and $u$ replaced by $D^{N-1}u$), we have 
\begin{equation}\label{stima4}
\partial^\alpha u\in X^{m+d+N-1-\epsilon},\quad |\alpha|=N
\end{equation}
for every $\epsilon>0$.\par
Now, let $\chi\in C^\infty_0(\rd)$, $\chi=1$ in a neighborhood of the origin. We introduce commutators in the equation $p(D)u=F(u)$ and apply $p(D)^{-1}$; we have 
\begin{multline}\label{stima5}
x^\beta\partial^\alpha u
=\underbrace{p(D)^{-1}[(\chi p)(D),x^\beta]\partial^\alpha u}_{=:q_1(D)u}
+\underbrace{p(D)^{-1}[((1-\chi) p)(D),x^\beta]\partial^{\alpha} u}_{=:q_2(D)u}\\
+p(D)^{-1}(x^\beta \partial^\alpha F(u)).
\end{multline}
Consider first the last term. By the Fa\`a di Bruno formula and the interpolation inequalities we have 
\begin{equation}\label{eqq}
\|D^N F(u)\|_{L^\infty(I)}\leq C\sum_{1\leq \nu\leq N} \|F'\|_{C^{\nu-1}}\|u\|_{L^\infty(I)}^{\nu-1}\|D^N u\|_{L^\infty(I)}.
\end{equation}
cf. \cite[Formula (3.1.9)]{taylor}. In this formula, we use the same notation as in Proposition \ref{interp}, and the norm of $F'$ is meant on the range of $u$.\par We now estimate each term by observing that, for $\nu>1$ (integer)  we have $|u|^{\nu-1}\in X^{m+d}$, which implies $\langle x\rangle^{1+\epsilon} |u|^{\nu-1}\in L^\infty$ if $\epsilon<m$; when $\nu=1$ instead $|F'(u)|\leq C|u|\in X^{m+d}$, so that $\langle x\rangle^{1+\epsilon} F'(u)\in L^\infty$ if $\epsilon<m$. On the other hand, by  \eqref{stima4}, we also have $\partial^\alpha u\in X^{m+d+N-1-\epsilon}$ if $|\alpha|=N$. We deduce by \eqref{eqq} that $D^N F(u)\in X^{m+d+N}$. By Proposition \ref{pro5} we get $p(D)^{-1}(x^\beta\partial^\alpha F(u))\in X^{m+d}$. \par
We now prove the same conclusion for the first two terms in the right-hand side of \eqref{stima5}. \par
Consider $q_1(D)u$. By the symbolic calculus we can write
\[
[(\chi p)(D),x^\beta]\partial^\alpha u=-\sum_{0\not=\gamma\leq\beta}i^{|\gamma|}\binom{\beta}{\gamma}\big(\partial^\gamma_\xi (\chi p)\big)(D)(x^{\beta-\gamma}\partial^\alpha u),
\]
where the derivatives of $\chi p$ in the right-hand side are meant in the sense of distributions.
By the inverse Leibniz formula\footnote{Namely, 
\[
x^\beta\partial^\alpha
u(x)=\sum_{\gamma\leq\beta,\,\gamma\leq\alpha}\frac{(-1)^{|\gamma|}\beta!}{(\beta-\gamma)!}
\binom{\alpha}{\gamma}\partial^{\alpha-\gamma}(x^{\beta-\gamma}u(x)).
\]
} we then obtain (since for $|\beta|\leq|\alpha|$)
\[
[(\chi p)(D),x^\beta]\partial^\alpha u=\sum_{0\not=\gamma\leq\beta}\sum_{\tilde{\alpha},\tilde{\beta}:|\tilde{\beta}|\leq|\tilde{\alpha}|<|\alpha|}  
C_{\alpha,\beta,\tilde{\alpha},\tilde{\beta}}
(\partial^{\ga}_\xi (\chi p))(D)\partial^{\tilde{\gamma}} (x^{\tilde{\beta}}\partial^{\tilde{\alpha}}u).
\]
where $\tilde{\gamma}$ is a suitable multi-index depending on $\alpha,\beta,\tilde{\alpha},\tilde{\beta},\gamma$, with $|\tilde{\gamma}|=|\gamma|$, and $C_{\alpha,\beta,\tilde{\alpha},\tilde{\beta}}$ are suitable constants. 
 \par
Now, by the inductive hypothesis we have $x^{\tilde{\beta}}\partial^{\tilde{\alpha}}u\in X^{m+d}$. The multiplier $\partial^{\ga}_\xi (\chi p)(D)\partial^{\tilde{\gamma}}$ has as symbol a sum of homogeneous functions, multiplied by cut-off functions, and having the same orders of the terms appearing in $p(\xi)$. In particular, the non-smooth terms have order at least $m$. Arguing as in the first part of the proof of Proposition \ref{pro5} we then deduce that such a symbol belong to $\Fur L^\infty_{v_{m+d}}$, as well as $p(\xi)^{-1}$ (Proposition \ref{pro5}), and therefore $q_1(D)$ maps $X^{m+d}$ into itself. This yields $q_1(D)u\in X^{m+d}$. \par
Consider now the term $q_2(D)u$ in \eqref{stima5}. Again by the symbolic calculus we can write
\begin{align}
q_2(D)u&= p(D)^{-1}[((1-\chi) p)(D),x^\beta]\partial^\alpha u\nonumber\\
&=-\sum_{0\not=\gamma\leq\beta}i^{|\gamma|}\binom{\beta}{\gamma}p(D)^{-1} \big(\partial^\gamma_\xi ((1-\chi)p)\big)(D)(x^{\beta-\gamma}\partial^\alpha u).\label{qpq}
\end{align}
Now, the multiplier
\[
p(D)^{-1}(\partial^{\ga}_\xi ((1-\chi) p))(D)
\]
has a symbol smooth in $\rd$ of order $-M+M-|\gamma|<0$, hence it maps $X^s\to X^s$ for $0\leq s\leq m+d$ by the same arguments as at the end of the proof of Proposition \ref{pro5}. Using \eqref{stima4} and $|\beta-\gamma|\leq N-1$ we see that $x^{\beta-\gamma}\partial^\alpha u\in X^{m+d-\epsilon}$ and therefore $q_2(D) u\in X^{m+d-\epsilon}$. Hence $u\in X^{m+d+N-\epsilon}$. We can then use this information to obtain in \eqref{qpq} $x^{\beta-\gamma}\partial^\alpha u\in X^{m+d+1-\epsilon}\subset X^{m+d}$ (for $\epsilon<1$) and therefore $q_2(D)u\in X^{m+d}$. This concludes the proof of \eqref{stima3}.
\vskip0.2cm
Theorem \ref{mainteo2bis} is a direct consequence of Theorem \ref{mainteo2}. We leave the proof to the reader.

\section{Analyticity of the solution: proof of Theorem \ref{mainteo3}}
We already know from the proof of Theorem \ref{mainteo2} that $u\in H^s$ for every $s\in\R$. Fix $s>d/2$, $\delta >0, N \in \N$ and consider the energy norms
\[
\mathcal{E}_N^\delta [u]=\sum_{|\alpha|\leq N}\frac{\delta^{|\alpha|}}{\alpha!}\|\partial^\alpha u\|_{H^s}.
\]
It is sufficient to prove that there exists $\delta>0$ such that the sequence $\mathcal{E}_N^\delta [u]$, $N=0,1,\ldots,$ is bounded. 
\par We can assume without loss of generality that $F(u)=u^k$ for some integer $k \geq 2.$ Starting from the identity $u=p(D)^{-1}(u^k)$ and differentiating both members we obtain, for any $\alpha \in \N^d \setminus \{0\}$:
$$\partial^\alpha u = p(D)^{-1} \partial^\alpha (u^k) = p(D)^{-1} \circ \partial_j (\partial^{\alpha -e_j}(u^k))$$
for some $j=j_\alpha \in \{1,\ldots, d\}.$\par
Now, we observe that the ellipticity condition \eqref{ellitticita} and $M>0$ imply that for every $\tau>0$ there exists a constant $C_\tau>0$ such that
\[
|p(\xi)^{-1}|\leq\tau+C_\tau\langle \xi\rangle^{-1},
\]
so that
\[
\|p(D)^{-1}\circ \partial_j v\|_{H^s}\leq \tau  \|\partial_j v\|_{H^s}+C_\tau\|v\|_{H^s}.
\]
Hence we get
$$
\|\partial^\alpha u\|_s \leq \tau \|\partial^\alpha u^k\|_{H^s}+C_\tau\|\partial^{\alpha-e_j} u^k\|_{H^s}.$$
Applying Leibniz' rule we have, since $s>d/2$,

\begin{multline*}
\frac{\|\partial^\alpha u\|_s}{\alpha!}\leq C\tau\|u\|_{H^s}^{k-1}\frac{\|\partial^\alpha u\|_{H^s}}{\alpha!}
+C\tau \sum_{\alpha_1 + \ldots +\alpha_k =\alpha \atop \alpha_j\not=\alpha\,\forall j} \frac{\|\partial^{\alpha_1} u\|_s}{\alpha_1!} \cdot \ldots \cdot \frac{\|\partial^{\alpha_k} u\|_s}{\alpha_k!}\\
+C C_\tau\sum_{\alpha_1 + \ldots +\alpha_k =\alpha-e_j} \frac{\|\partial^{\alpha_1} u\|_s}{\alpha_1!} \cdot \ldots \cdot \frac{\|\partial^{\alpha_k} u\|_s}{\alpha_k!}
\end{multline*}
for some constant $C>0$ depending only on $s,k,d$.\par

Hence, multiplying by $\delta^{|\alpha|}$ and summing up for $|\alpha| \leq N$ it follows that
\begin{eqnarray*}
\mathcal{E}_N^{\delta}[u]\leq  \|u\|_s  + C\tau\|u\|_{H^s}^{k-1}\mathcal{E}_N^{\delta}[u]+    C'(\tau+\delta C_\tau)\ (\mathcal{E}_{N-1}^{\delta}[u])^k,
\end{eqnarray*} 
for some new constant $C'>0$ depending only on $d,k$.\par
If we choose $\tau$ and then $\delta$ small enough, iterating this estimate it is easy to prove that $\mathcal{E}_N^{\delta}[u]$ is bounded with respect to $N$. Theorem \ref{mainteo3} is then proved.

\vskip0.3cm
\textbf{Acknowledgement.} The authors wish to thank Professors Todor Gramchev and Luigi Rodino for some helpful discussions and comments.

\end{document}